\documentclass[12 pt]{amsart}

\usepackage{accents}    
\usepackage{appendix}
\usepackage{amsfonts}
\usepackage{amsmath}
\usepackage{amssymb}	     
\usepackage{amsthm} 
\usepackage{lipsum}   
\usepackage{array,booktabs,multirow}
\usepackage{braket}   
\usepackage{cite}  
\usepackage{dsfont}
\usepackage{mathalfa}        
\usepackage[shortlabels]{enumitem} 
\usepackage{etoolbox}            
\usepackage{float}
\usepackage[hang, flushmargin]{footmisc}   
\usepackage{latexsym}  
\usepackage{needspace}       
\usepackage{tikz} 
\usepackage[pdfencoding=auto, psdextra]{hyperref}
\usetikzlibrary{matrix,arrows}
\usepackage{soul}   
   
\usepackage{filecontents}

\theoremstyle{plain}
\newtheorem{thm}{Theorem}[section]
\newtheorem{cor}[thm]{Corollary}
\newtheorem{lem}[thm]{Lemma}

\theoremstyle{definition}

\newtheorem{defn}[thm]{Definition}

\newtheorem{note}[thm]{Note}

\theoremstyle{remark}

\AtBeginEnvironment{lem}{\Needspace{3\baselineskip}}
\AtBeginEnvironment{thm}{\Needspace{3\baselineskip}}
\AtBeginEnvironment{prop}{\Needspace{3\baselineskip}}
\AtBeginEnvironment{cor}{\Needspace{3\baselineskip}}

\setlist[enumerate,1]{leftmargin=2 em}

\def\N{\mathbb N} 
\def\F{\mathbb F}
\def\Z{\mathbb Z}
\def\Ce{\mathfrak{C}}
\def\sl{\mathfrak{sl}_2}

\def\U{U(\mathfrak{sl}_2)}

\def\FU{\F[a,b,c]\otimes \U}
\def\comm{1\otimes \Lambda+a(a+1)\otimes 1+b(b+1)\otimes 1+c(c+1)\otimes 1}

\title[The universal enveloping algebra of $\mathfrak{sl}_2$ and the Racah algebra]{The universal enveloping algebra of $\mathfrak{sl}_2$ and\\ the Racah algebra}

\author{Sarah Bockting-Conrad}
\address{
Sarah Bockting-Conrad\\
Department of Mathematical Sciences\\
DePaul University\\
Chicago, Illinois, USA}
\email{sarah.bockting@depaul.edu}

\author{Hau-Wen Huang}
\address{
Hau-Wen Huang\\
Department of Mathematics\\
National Central University\\
Chung-Li 32001 Taiwan
}
\email{hauwenh@math.ncu.edu.tw}

\thanks{The research of the second author is supported by the Ministry of Science and Technology of Taiwan under the project MOST 106-2628-M-008-001-MY4}

\begin{document}

\begin{abstract}
Let $\F$ denote a field with ${\rm char\,}\F\not=2$. The Racah algebra $\Re$ is the unital associative $\F$-algebra defined by generators and relations in the following way.  The generators are $A$, $B$, $C$, $D$. The relations assert that 
\begin{equation*}
[A,B]=[B,C]=[C,A]=2D
\end{equation*} 
and each of the elements
\begin{gather*}
\alpha=[A,D]+AC-BA,
\qquad
\beta=[B,D]+BA-CB,
\qquad
\gamma=[C,D]+CB-AC
\end{gather*}
is central in $\Re$. Additionally the element $\delta=A+B+C$ is central in $\Re$. 
 In this paper we explore the relationship between the Racah algebra $\Re$ and the universal enveloping algebra $\U$. Let $a,b,c$ denote mutually commuting indeterminates.  We show that there exists a unique $\F$-algebra homomorphism $\natural:\Re\to\F[a,b,c]\otimes_\F \U$ that sends
\begin{eqnarray*}
A  &\mapsto&
a(a+1)\otimes 1+(b-c-a)\otimes x+(a+b-c+1)\otimes y-1\otimes xy,
\\
B  &\mapsto&
b(b+1)\otimes  1+(c-a-b)\otimes y+(b+c-a+1)\otimes z-1\otimes yz,
\\
C  &\mapsto&
c(c+1)\otimes  1+(a-b-c)\otimes z+(c+a-b+1)\otimes x-1\otimes zx,
\\
D  &\mapsto&
1\otimes  (zyx+zx)+
(c+b(c+a-b))\otimes x
+(a+c(a+b-c))\otimes y
\\
&&
\qquad+(b+a(b+c-a))\otimes z +\,(b-c)\otimes xy+(c-a)\otimes yz+(a-b)\otimes zx,
\end{eqnarray*}
where $x,y,z$ are the equitable generators for $\U$.  
We additionally give the images of $\alpha,\beta,\gamma,\delta,$ and certain Casimir elements of $\Re$ under $\natural$.  
We also show that the map $\natural$ is an injection and thus provides an embedding of $\Re$ into $\F[a,b,c]\otimes\U$.
We use the injection to show that $\Re$ contains no zero divisors. 

\end{abstract}

\maketitle
 
\section{Introduction}  
	In this paper, we consider a universal analogue of the classical Racah algebras which appear in a variety of contexts. 
		The Racah algebras were first discovered in the study of the coupling problem for three angular momenta \cite{Levy1965}, though the algebras were not referred to as the Racah algebras until much later.  
	Approximately twenty years later, Granovski{\u\i} and Zhedanov rediscovered the Racah algebras in an alternate presentation, now referred to as the standard presentation,
and gave a realization for the Racah algebras in terms of the intermediate Casimir operators of the Lie algebra $\mathfrak{su}(2)$  \cite{zhedanov1988}.  
	The algebras were again rediscovered in \cite{LTRacah}, but this time in the context of Leonard triples of Racah type.    
	In addition to the works cited above, connections have been found between the Racah algebras and a host of other areas including the Askey scheme of classical orthogonal polynomials, superintegrable models, and the Racah problem for the Lie algebra $\mathfrak{su}(1,1)$ \cite{Galbert, integrable2014-1,integrable2014-2, Racahproblem2014, R&BI2015,Kalnins1,Kalnins2, Kalnins3,Kalnins4,zhedanov1988, quadratic1991,quadratic1992, quadratic1989, hiddensymmetry1,hiddensymmetry2,hiddensymmetry3}.

For the rest of the present paper, we let $\F$ denote a field with ${\rm char\,}\F\not=2$.  Let $\Z$ denote the set of all integers and let $\N$ denote the set of all nonnegative integers. The unadorned tensor products are meant to be over $\F$.
 When we discuss an algebra, we mean a unital associative algebra. When we discuss a subalgebra, we assume that it has the same unit 
 as the parent algebra. 
 
 We let $\Re$ denote the  
  $\F$-algebra defined by generators and relations in the following way.  
  The generators are $A$, $B$, $C$, $D$.  
The relations assert that 
\begin{equation*}
[A,B]=[B,C]=[C,A]=2D
\end{equation*}
and that each of the elements
\begin{equation*}
\alpha=[A,D]+AC-BA,
\qquad
\beta=[B,D]+BA-CB, 
\qquad
\gamma=[C,D]+CB-AC
\end{equation*}
is central in $\Re$.
It follows from the above definition that the element $\delta=A+B+C$ is also central in $\Re$. The algebra $\Re$ is a universal analogue of the original Racah algebras and is referred to as the {\it Racah algebra} hereafter \cite{Levy1965, R&BI2015}.	

We now mention some earlier work concerning this algebra. 
In \cite{Huang:R<BI,Huang:racah-DAHA}, the second author explores the relationship between $\Re$ and the additive double affine Hecke algebra $\mathfrak{H}$ of type $(C_1^\vee,C_1)$.  The author shows that $\Re$ is isomorphic to a subalgebra of $\mathfrak H$ and leverages this fact along with the information in \cite{SH:2019-1, HHL:BImodules} to classify the lattices of $\Re$-submodules of finite-dimensional irreducible $\mathfrak{H}$-modules, provided that $\F$ is algebraically closed with characteristic zero.

	In \cite{racah1}, the present authors discuss the algebra $\Re$ and investigate its Casimir class which we now define.  Let $\Ce$ denote the commutative subalgebra of $\Re$ generated by $\alpha$, $\beta$, $\gamma$, $\delta$.  
Inspired by the work by Genest--Vinet--Zhedanov in \cite[Section 2]{integrable2014-2}, we consider the following coset of $\Ce$ in $\Re$:
\begin{equation*}
D^2+A^2+B^2
+\frac{(\delta+2)\{A,B\}-\{A^2,B\}-\{A,B^2\}}{2}
+A (\beta-\delta)
+B (\delta-\alpha)+\Ce,
\label{eq:abs}
\end{equation*}
where $\{\,,\,\}$ stands for the anticommutator.  
We refer to this coset as the \emph{Casimir class} of $\Re$ and call an element of the Casimir class a \emph{Casimir element} of $\Re$.  
In \cite{racah1}, we show that every Casimir element $\Omega$ of $\Re$ is central in $\Re$.  Further, we show that $\Omega$ is algebraically independent over $\Ce$ and  that whenever ${\rm char\,}\F=0$, the center of $\Re$ is precisely $\Ce[\Omega]$.  
Some of the results in \cite{racah1} focus on a set of Casimir elements $\{\Omega_A,\Omega_B,\Omega_C\}$ which is invariant under a certain $D_6$-action on $\Re$.  
In particular, we show that for $\Omega\in\{\Omega_A,\Omega_B,\Omega_C\}$, the elements
\begin{gather*}
A^i D^j B^k
\Omega^\ell
\alpha^r
\delta^s
\beta^t
\qquad \quad
\hbox{for all $i,k,\ell,r,s,t\in \N$ and $j\in \{0,1\}$}
\end{gather*}
are a basis for the $\F$-vector space $\Re$.
The present paper builds upon this work.

We now give an overview of the present paper.  In this paper we explore the relationship between the Racah algebra $\Re$ and the universal enveloping algebra $\U$.  Along this vein there are three main results.  First we show that there exists a unique algebra homomorphism $\natural:\Re\to\FU$ that sends
\begin{eqnarray*}
A  &\mapsto&
a(a+1)\otimes 1+(b-c-a)\otimes x+(a+b-c+1)\otimes y-1\otimes xy,
\\
B  &\mapsto&
b(b+1)\otimes  1+(c-a-b)\otimes y+(b+c-a+1)\otimes z-1\otimes yz,
\\
C  &\mapsto&
c(c+1)\otimes  1+(a-b-c)\otimes z+(c+a-b+1)\otimes x-1\otimes zx,
\\
D  &\mapsto&
1\otimes  (zyx+zx)+
(c+b(c+a-b))\otimes x
+(a+c(a+b-c))\otimes y
\\
&&
\qquad+(b+a(b+c-a))\otimes z +\,(b-c)\otimes xy+(c-a)\otimes yz+(a-b)\otimes zx.\nonumber
\end{eqnarray*}
Here $a,b,c$ are mutually commuting indeterminates, $\F[a,b,c]$ denotes the $\F$-algebra of polynomials in $a,b,c$ that have all coefficients in $\F$, and $x,y,z$ are the equitable generators for $\U$.   
Further, the homomorphism $\natural$ sends
\begin{eqnarray*}
\alpha  &\mapsto&
(1\otimes \Lambda- a(a+1)\otimes 1)\cdot (b(b+1)\otimes 1-c(c+1)\otimes 1),
\\
\beta  &\mapsto&
(1\otimes \Lambda- b(b+1)\otimes 1)\cdot (c(c+1)\otimes 1-a(a+1)\otimes 1),
\\
\gamma  &\mapsto&
(1\otimes \Lambda- c(c+1)\otimes 1)\cdot (a(a+1)\otimes 1-b(b+1)\otimes 1),
\\
\delta  &\mapsto&
1\otimes\Lambda+a(a+1)\otimes 1+b(b+1)\otimes 1+c(c+1)\otimes 1,
\end{eqnarray*}
where $\Lambda$ is the normalized Casimir element of $\U$.  The second main result is that the images of $\Omega_A,\Omega_B,\Omega_C$ under $\natural$ are 
\begin{align*}
&(1\otimes \Lambda+a(a+1)\otimes 1-b(b+1)\otimes 1-c(c+1)\otimes 1)\cdot (a(a+1)\otimes \Lambda-b(b+1)c(c+1)\otimes 1)
\\
&\qquad-(1\otimes \Lambda+a(a+1)\otimes 1)\cdot(b(b+1)\otimes 1+c(c+1)\otimes 1), \\
&(1\otimes \Lambda+b(b+1)\otimes 1-c(c+1)\otimes 1-a(a+1)\otimes 1)\cdot (b(b+1)\otimes \Lambda-c(c+1)a(a+1)\otimes 1)
\\
&\qquad-(1\otimes \Lambda+b(b+1)\otimes 1)\cdot(c(c+1)\otimes 1+a(a+1)\otimes 1),\\
&(1\otimes \Lambda+c(c+1)\otimes 1-a(a+1)\otimes 1-b(b+1)\otimes 1)\cdot (c(c+1)\otimes \Lambda-a(a+1)b(b+1)\otimes 1)
\\
&\qquad-(1\otimes \Lambda+c(c+1)\otimes 1)\cdot(a(a+1)\otimes 1+b(b+1)\otimes 1),
\end{align*}
respectively.  
The third main result is that the homomorphism $\natural$ is an injection and thus provides an embedding of $\Re$ into $\F[a,b,c]\otimes\U$.  
 We use the injection to show that $\Re$ contains no zero divisors.
  
The present paper is organized as follows.
In Section \ref{section:sl}, we recall the Lie algebra $\mathfrak{sl}_2$ and its universal enveloping algebra $U(\mathfrak{sl}_2)$.
In Section \ref{section:comms}, we give some results concerning commutators of elements in $\U$. 
In Section \ref{section:racah}, we recall some facts concerning the Racah algebra $\Re$.
In Section \ref{section:embed}, we show the existence and uniqueness of the homomorphism $\natural:\Re\to\FU$.  
In Section \ref{section:Ugrad}, we discuss a Poincar\'{e}--Birkhoff--Witt basis for $U(\mathfrak{sl}_2)$ and its associated $\Z$-grading of $U(\mathfrak{sl}_2)$.  
In Section \ref{section:FUgrad}, we extend the $\Z$-grading of $\U$ to a $\Z$-grading of $\FU$ and describe the homogeneous components of certain elements of $\FU$.
In Section \ref{section:Casimages}, we discuss the images of the Casimir elements $\Omega_A,\Omega_B,\Omega_C$ under $\natural$. 
In Section \ref{section:algind}, we give some results concerning algebraic independence.
In Section \ref{section:injectivity}, we show that the map $\natural$ is injective.

\section{The Lie algebra $\sl$ and its universal enveloping algebra $\U$}\label{section:sl}
In this section we recall the Lie algebra $\mathfrak{sl}_2$ and its universal enveloping algebra $U(\mathfrak{sl}_2)$.  For additional information on $\mathfrak{sl}_2$ and $U(\mathfrak{sl}_2)$, see \cite{benkart2010, mazorchuk}.

Let $\mathfrak{sl}_2=\mathfrak{sl}_2(\F)$ denote the Lie algebra over $\F$ that has basis $E,F,H$ and Lie bracket
	\begin{gather}
		[H,E]=2E,\qquad\qquad [H,F]=-2F,\qquad\qquad [E,F]=H.
	\end{gather}
We refer to $E,F,H$ as the {\it natural} or {\it standard} basis for $\mathfrak{sl}_2$ and refer to the above presentation as the {\it natural} or {\it standard} presentation for $\mathfrak{sl}_2$.  In \cite{hartwig2}, the authors introduce an alternative basis $X,Y,Z$ of $\sl$ called the {\it equitable} basis which we describe in the lemma below.  For a comprehensive study of the equitable basis of $\mathfrak{sl}_2$, see \cite{benkart2010}.
\begin{lem}{\rm \cite[Lemmas 3.2 and 3.4(i)]{hartwig2}}\label{lemma:sl2equitable}
	The Lie algebra $\mathfrak{sl}_2$ is isomorphic to the Lie algebra over $\F$ that has basis $X,Y,Z$ and Lie bracket
	\begin{gather}
		[X,Y]= X+Y,	\qquad\qquad [Y,Z]= Y+Z, \qquad\qquad [Z,X]= Z+X.
	\end{gather}
	An isomorphism with the standard presentation for $\mathfrak{sl}_2$ is given by 
	\begin{gather}
		X\to -F-H/2,\qquad\qquad Y\to H/2,\qquad\qquad Z\to E-H/2.
	\end{gather}
	The inverse of this isomorphism is given by 
	\begin{gather}
		E\to Y+Z,\qquad\qquad F\to-X-Y,\qquad\qquad H\to 2Y.
	\end{gather}
\end{lem}

\begin{note} For notational convenience, for the rest of the present paper, we will identify the copy of $\mathfrak{sl}_2$ given in the original definition with the copy given in Lemma \ref{lemma:sl2equitable}, via the isomorphism given in Lemma \ref{lemma:sl2equitable}.  
\end{note}

We now describe the universal analogue of $\mathfrak{sl}_2$.  By the universal enveloping algebra $\U$ of $\mathfrak{sl}_2$, we mean the  
 $\F$-algebra with generators $e$, $f$, $h$ and relations
\begin{gather}
he-eh=2e, \qquad
hf-fh=-2f, \qquad  
ef-fe= h.\label{eq:Udef}
\end{gather}

There is a natural embedding of $\sl$ into $\U$ via which we associate the basis elements $E,F,H$ of $\sl$ with the generators $e,f,h$ of $\U$.
This association, along with the equitable basis for $\sl$, gives rise to what is known as the {\it equitable presentation} for $\U$.  
  Let
\begin{gather}
x=-f-\frac{h}{2},
\qquad
y=\frac{h}{2},
\qquad
z=e-\frac{h}{2}.\label{eq:Uequit}
\end{gather}
Solving these equations for $e$, $f$, $h$  yields that
\begin{gather}
e=y+z,
\qquad
f=-x-y,
\qquad
h=2y.\label{eq:isowhat}
\end{gather}
We see that $x$, $y$, $z$ generate $\U$.
The elements $x,y,z$ are called the {\it equitable generators} of $\U$ and are subject to the following relations:
\begin{gather}\label{e:eqrelU}
xy-yx=x+y,
\qquad
yz-zy=y+z,
\qquad
zx-xz=z+x.
\end{gather}

We conclude the section by mentioning a notable element of $\U$ which will appear later in our discussion of the Racah algebra.
We let \begin{eqnarray}
\Lambda &=& ef+\frac{h(h-2)}{4}\label{eq:Lambdadef}
\end{eqnarray}
and call $\Lambda$ the {\it normalized Casimir element} of
$U({\mathfrak{sl}_2})$ \cite[p. 11]{mazorchuk}.  
With respect to the equiatble presentation, the normalized Casimir element of $\U$ is given by
\begin{equation}
\Lambda=-\frac{xy+yz+zx+yx+zy+xz}{2}. 
\end{equation}
It is quick to verify that $\Lambda$ is central in $\U$.

\section{Some commutators in $\U$}\label{section:comms}
Recall the universal enveloping algebra $\U$ discussed in Section \ref{section:sl}.  In this section, we give some results concerning commutators of certain elements in $\U$.  These results will be used in Section \ref{section:embed} when proving the first main result.  Throughout the rest of the paper, we let $[\,,\,]$ denote the usual commutator.   

\begin{lem}\label{lemma:[x,xy]}
The following equations hold in $\U$:
\begin{alignat*}{3}
[x,xy]&=x^2+xy,
\qquad
[x,yz]&&=xz-yx, 
\qquad
[x,zx]&&=-x^2-zx,
\\
[y,yz] &= y^2+yz,
\qquad
[y,zx] &&= yx-zy, 
\qquad
[y,xy] &&= -y^2-xy,
\\
[z,zx] &= z^2+zx,
\qquad
[z,xy] &&=zy-xz, 
\qquad
[z,yz] &&= -z^2-yz.
\end{alignat*}
\end{lem}
\begin{proof}
	This result follows from (\ref{e:eqrelU}). 
\end{proof}

\begin{lem}\label{lemma:[xy,yz]}
The following equations hold in $\U$:
\begin{gather*}
[xy,yz]=2xyz+y^2-xz,  
\\
[yz,zx]=2yzx+z^2-yx, 
\\
[zx,xy]=2zxy+x^2-zy. 
\end{gather*}
\end{lem}
\begin{proof}
	This result follows from (\ref{e:eqrelU}) and Lemma \ref{lemma:[x,xy]}. 
\end{proof}

\begin{lem}\label{lemma:w}
The following elements of $\U$ coincide:
\begin{align}
&zyx+zx,
&&zxy-zy,
&&yzx-yx,
\label{eq:w1}\\
&xzy+xy,
&&yxz+yz,
&&xyz-xz.
\label{eq:w2}
\end{align}
\end{lem}
\begin{proof}
	By \eqref{e:eqrelU}, we see that 
	$zyx+zx=z(xy-x-y)+zx=zxy-zy$
	and also that
	$zyx+zx=(yz-y-z)x+zx=yzx-yx$.
	So the expressions in \eqref{eq:w1} are all equal.  The remaining equalities can be similarly obtained.
\end{proof}

\begin{defn}\label{def:w}
Let $w$ denote the common element of $\U$ given in Lemma \ref{lemma:w}.
\end{defn}

\begin{lem}\label{lemma:[w,x]}
The following equations hold in $\U$:
\begin{gather}
[w,x]=xyx-xzx,\label{eq:wx}
\\
[w,y]=yzy-yxy,\label{eq:wy}
\\
[w,z]=zxz-zyz.\label{eq:wz}
\end{gather}
\end{lem}
\begin{proof}
	We first show \eqref{eq:wx}.  By Lemma \ref{lemma:w} and Definition \ref{def:w}, $w=xyz-xz$ and $w=yzx-yx$.  Thus, 
	\begin{align*}
		[w,x]&=wx-xw\\
		&=(xyz-xz)x-x(yzx-yx)\\
		&=xyx-xzx.
	\end{align*}
	Equations \eqref{eq:wy},\eqref{eq:wz} can be shown similarly.
\end{proof}

\begin{lem}\label{lemma:[w,xy]}
The following equations hold in $\U$:
\begin{align*}
[w,xy]&=
yzxy - xyzx + xyx - yxy,
\\
[w,yz]&=
zxyz - yzxy + yzy - zyz,
\\
[w,zx]&=
xyzx - zxyz + zxz - xzx.
\end{align*}
\end{lem}
\begin{proof}
	This result follows from (\ref{e:eqrelU}) and Lemma \ref{lemma:[w,x]}. 
\end{proof}

\section{The Racah algebra $\Re$}\label{section:racah}

In this section we recall some facts about the Racah algebra $\Re$ from \cite{racah1} and define a bilinear form on $\Re$ which we will use when proving the third main result in Section \ref{section:injectivity}.

\begin{defn}\cite[Definition 3.1]{racah1}\label{def:Re}
Define an $\F$-algebra $\Re$ by generators and relations in the following way. The generators are $A$, $B$, $C$, $D$. The relations assert that
\begin{gather}\label{r:D}
[A,B]=[B,C]=[C,A]=2D
\end{gather}
and each of
\begin{gather*}
[A,D]+AC-BA,
\qquad
[B,D]+BA-CB,
\qquad
[C,D]+CB-AC
\end{gather*}
is central in $\Re$.  We call $\Re$ the {\it Racah algebra}.
\end{defn}

For notational convenience, we let 
\begin{eqnarray}
\alpha &=& [A,D]+AC-BA, \label{r:alpha} \\
\beta &=& [B,D]+BA-CB,\label{r:beta}\\
\gamma &=& [C,D]+CB-AC,\label{r:gamma}\\
\delta &=& A+B+C. \label{r:delta}
\end{eqnarray}
Each of $\alpha, \beta, \gamma,\delta$ is central in $\Re$ by \cite[Lemma 3.2]{racah1}.  

Let $\Ce$ denote the commutative subalgebra of $\Re$ generated by $\alpha$, $\beta$, $\gamma$, $\delta$.  By \cite[Lemma 5.2]{racah1}, 
the elements 
 $\{ \alpha^r\delta^s\beta^t | 
r,s,t\in \N \}$ 
 form a basis for the  $\F$-vector space $\Ce$.  Let the curly bracket $\{\,,\,\}$ stand for the anticommutator. 
We call the coset 
\begin{gather*}
A^2+B^2+D^2
+\frac{\delta+2}{2}\{A,B\}
-\frac{\{A^2,B\}}{2}
-\frac{\{A,B^2\}}{2}
+A (\beta-\delta)
-B (\alpha-\delta)
+\Ce
\end{gather*}
the {\it Casimir class} of $\Re$. An element $\Omega$ of $\Re$ is called a {\it Casimir element} if $\Omega$ lies in the Casimir class of $\Re$. 
Observe that the difference of any two Casimir elements of $\Re$ can be expressed as a polynomial in $\alpha, \beta, \delta$. 
By \cite[Proposition 6.7]{racah1}, each element of the Casimir class is central in $\Re$.  By \cite[Lemma 7.11]{racah1}, if ${\rm char\,} \F=0$ and $\Omega$ is any Casimir element of $\Re$, then $Z(\Re)=C[\Omega]$.

Define three elements $\Omega_A$, $\Omega_B$, $\Omega_C$ of $\Re$ by
\begin{eqnarray}
\Omega_A
&=&
D^2
+
\frac{B A C
+C A B}{2}
+ A^2
+B \gamma
-C \beta
-A \delta,
\label{eq:CasA}
\\
\Omega_B
&=&
D^2
+
\frac{C B A
+A B C}{2}
+ B^2
+C \alpha 
-A \gamma
-B\delta,
\label{eq:CasB}
\\
\Omega_C
&=&
D^2
+
\frac{A C B
+B C A}{2}
+ C^2 
+A \beta 
-B\alpha
-C\delta.
\label{eq:CasC}
\end{eqnarray}
By \cite[Proposition 6.4]{racah1}, 
each of $\Omega_A$, $\Omega_B$, $\Omega_C$ is a Casimir element of $\Re$.
By \cite[Lemma 7.1 \& Theorem 7.5]{racah1}, for $\Omega\in \{\Omega_A,\Omega_B,\Omega_C\}$, 
 the elements
\begin{gather}
A^i D^j B^k
\Omega^\ell
\alpha^r
\delta^s
\beta^t
\qquad \quad
\hbox{for all $i,k,\ell,r,s,t\in \N$ and $j\in \{0,1\}$}\label{eq:basisURA2}
\end{gather}
are a basis for the $\F$-vector space $\Re$.
Later in this paper we will be considering elements of $\Re$ as linear combinations of these basis elements.  To aid in our discussion, we define a bilinear form $\langle,\rangle:\Re\times\Re\to\F$ such that  $\langle u,v\rangle=\delta_{u,v}$ for all $u,v$ in the basis (\ref{eq:basisURA2}).  Observe that the basis (\ref{eq:basisURA2}) is orthnormal with respect to $\langle,\rangle$.  We see that for $u\in\Re$, 
\begin{equation*}
u=\sum_{\substack{i,k,\ell,r,s,t\in \N,\\ j\in \{0,1\}}}\langle u, A^i D^j B^k
\Omega^\ell
\alpha^r
\delta^s
\beta^t
\rangle A^i D^j B^k
\Omega^\ell
\alpha^r
\delta^s
\beta^t.
\end{equation*}
Note that there are only finitely many nonzero summands in this sum.

\section{A homomorphism from $\Re$ into $\F[a,b,c]\otimes \U$}\label{section:embed}

Recall from the Introduction that $a,b,c$ are three mutually commuting indeterminates and that $\F[a,b,c]$ denotes the $\F$-algebra of polynomials in $a,b,c$ that have all coefficients in $\F$. 
In this section, we prove the first theorem concerning the map $\natural:\Re\to\FU$. 

\begin{thm}\label{thm:natural}
There exists a unique $\F$-algebra homomorphism $\natural:\Re\to \F[a,b,c]\otimes \U$ that sends
\begin{eqnarray}
A  &\mapsto&
a(a+1)\otimes 1+(b-c-a)\otimes x+(a+b-c+1)\otimes y-1\otimes xy,
\label{e:Anatural}
\\
B  &\mapsto&
b(b+1)\otimes  1+(c-a-b)\otimes y+(b+c-a+1)\otimes z-1\otimes yz,
\label{e:Bnatural}
\\
C  &\mapsto&
c(c+1)\otimes  1+(a-b-c)\otimes z+(c+a-b+1)\otimes x-1\otimes zx,
\label{e:Cnatural}
\\
D  &\mapsto&
1\otimes  w+
(c+b(c+a-b))\otimes x
+(a+c(a+b-c))\otimes y
\label{e:Dnatural}
\\
&&
\qquad+(b+a(b+c-a))\otimes z +\,(b-c)\otimes xy+(c-a)\otimes yz+(a-b)\otimes zx,\nonumber
\end{eqnarray}
where $x,y,z$ are the equitable generators for $\U$.  The homomorphism $\natural$ sends 
\begin{eqnarray}
\alpha  &\mapsto&
(1\otimes \Lambda- a(a+1)\otimes 1)\cdot (b(b+1)\otimes 1-c(c+1)\otimes 1),
\label{e:alphnatural}
\\
\beta  &\mapsto&
(1\otimes \Lambda- b(b+1)\otimes 1)\cdot (c(c+1)\otimes 1-a(a+1)\otimes 1),
\label{e:betanatural}
\\
\gamma  &\mapsto&
(1\otimes \Lambda- c(c+1)\otimes 1)\cdot (a(a+1)\otimes 1-b(b+1)\otimes 1),
\label{e:gamnatural}
\\
\delta  &\mapsto&
1\otimes\Lambda+a(a+1)\otimes 1+b(b+1)\otimes 1+c(c+1)\otimes 1,
\label{e:delnatural}
\end{eqnarray}
where $\Lambda$ denotes the normalized Casimir element of $\U$.
\end{thm}

For notational convenience, throughout the rest of the paper we let  $A^\natural$, $B^\natural$, $C^\natural$, $D^\natural$, $\alpha^\natural$, $\beta^\natural$, $\gamma^\natural$, $\delta^\natural$ denote the expressions occurring on the right-hand side of (\ref{e:Anatural})--(\ref{e:delnatural}) respectively.
Note that since $\Lambda$ is central in $\U$, each of $\alpha^\natural$, $\beta^\natural$, $\gamma^\natural$, $\delta^\natural$ is central in $\FU$.

\begin{proof}[Proof of Theorem \ref{thm:natural}:]
 To establish the existence of the homomorphism $\natural$, it suffices to verify that
\begin{equation}\label{e:natural[AB]=D}
[A^\natural,B^\natural]
=[B^\natural,C^\natural]
=[C^\natural,A^\natural]
=2 D^\natural
\end{equation}
and
\begin{eqnarray}
\alpha^\natural&=&[A^\natural,D^\natural]
+A^\natural C^\natural-B^\natural A^\natural,\label{eq:alphanat}
\\
\beta^\natural&=&[B^\natural,D^\natural]
+B^\natural A^\natural-C^\natural B^\natural,\label{eq:betanat}
\\
\gamma^\natural&=&[C^\natural,D^\natural]
+C^\natural B^\natural-A^\natural C^\natural,\label{eq:gammanat}
\\
\delta^\natural&=&A^\natural+B^\natural+C^\natural.\label{eq:deltanat}
\end{eqnarray}

Equations (\ref{e:natural[AB]=D})--(\ref{eq:deltanat}) can be verified through routine, though tedious, computations.  The general strategy is as follows.
View $\F[a,b,c]\otimes \U$ as the ring of polynomials in $a,b,c$ with all coefficients in 
 $\U$.  
For each side of a given equation, determine which monomials in $a,b,c$ occur with nonzero coefficients.  For a given monomial, use (\ref{e:eqrelU}) along with Lemmas \ref{lemma:[x,xy]}--\ref{lemma:w} and \ref{lemma:[w,x]}--\ref{lemma:[w,xy]} to show that the corresponding coefficients on each side of the equation are equal to one another. 

 To help illustrate this strategy, we now show how to verify that $[A^\natural,B^\natural]=2D^\natural$ as an example. 
 Observe that the terms with nonzero coefficients in $[A^\natural, B^\natural]$ or $2D^\natural$ are
\begin{equation*}
a^2,
\quad
b^2,
\quad
c^2,
\quad
ab,
\quad
bc, 
\quad
ca,
\quad
a,
\quad
b,
\quad
c,
\quad
1.
\end{equation*}
For each of these terms, we list the corresponding coefficients in $[A^\natural, B^\natural]$ and $2D^\natural$ in the table below, along with the  reason why the coefficients are equal.

\begin{table}[H]
\small
\centering
\extrarowheight=3pt
\begin{tabular}{c|c|c|c}
term
&coefficient in $[A^\natural, B^\natural]$
&coefficient in $2D^\natural$
&how to verify they are equal
\\
\hline
\hline

$a^2$
&$[x,y]-[y,z]-[z,x]$
& $-2z$

&\multirow{6}{*}{Use (\ref{e:eqrelU}).}
\\

$b^2$
&$[y,z]-[z,x]-[x,y]$
& $-2x$
\\

$c^2$
&$[z,x]-[x,y]-[y,z]$
& $-2y$
\\

\cline{1-3}

$ab$
&$2[z,x]$
& $2z+2x$
\\

$bc$
&$2[x,y]$
& $2x+2y$
\\

$ca$
&$2[y,z]$
& $2y+2z$
\\

\hline

$a$
&$[x-y,yz]-[y+z,xy]+[z,x]$
&$2zx-2yz+2y$
&\multirow{3}{*}{Use (\ref{e:eqrelU}) and Lemma \ref{lemma:[x,xy]}.}
\\

$b$
&$[z-y,xy]-[x+y,yz]-[z,x]+2[y,z]$
&$2xy-2zx+2z$
\\

$c$
&$[y+z,xy]+[x+y,yz]+[z,x]$
&$2yz-2xy+2x$
\\

\hline

$1$
&$[y,z]-[y,yz]+[z,xy]+[xy,yz]$
&$2w$
&
Use (\ref{e:eqrelU}) and 
Lemmas \ref{lemma:[x,xy]}--\ref{lemma:w}.

\end{tabular}
\end{table}

\noindent From the above table, it now follows that $[A^\natural,B^\natural]=2D^\natural$. The remaining equations in (\ref{e:natural[AB]=D})--(\ref{eq:deltanat}) can be verified through
 similar arguments.

Note that the homomorphism is unique since $A,B,C, D$ generate $\Re$. 
\end{proof}

\section{A $\Z$-grading of $\U$}\label{section:Ugrad}
In Section \ref{section:sl}, we recalled the universal enveloping algebra $\U$.  In this section, we recall a Poincar\'{e}--Birkhoff--Witt basis for $\U$ and discuss the associated $\Z$-grading of $\U$.  In the next section, we will extend the $\Z$-grading of $\U$ to a $\Z$-grading of $\FU$.  We will use these gradings when proving Theorems \ref{thm:Casimage} and \ref{thm:inj}. 

We first establish some terminology which we will use going forward.  
Let  $\mathcal A$ denote an $\F$-algebra and let  $\mathcal H, \mathcal K$ denote  $\F$-subspaces of  $\mathcal A$. By  $\mathcal H\cdot \mathcal K$, we mean the $\F$-subspace of  $\mathcal A$ spanned by $hk$ for all $h\in \mathcal H$ and $k\in \mathcal K$. For all $n\in \N$, the notation $\mathcal H^n$ stands for 
\begin{equation*}
 \underbrace{\mathcal H\cdot \mathcal H\cdots \mathcal H}_\text{$n$ copies}
\end{equation*}
For notational convenience, we define $\mathcal H^0$ to be $\F 1$, where $1$ is the unit of $\mathcal A$. 

Following \cite[p.202]{carter}, we define 
graded algebras as follows. 
We call the algebra $\mathcal A$ a {\it $\Z$-graded algebra} if there exists a decomposition 
\begin{equation*}
\mathcal A=\cdots \oplus \mathcal A_{-2}\oplus \mathcal A_{-1}\oplus\mathcal  A_0\oplus \mathcal A_1\oplus \mathcal A_2\oplus\cdots
\end{equation*}
of $\mathcal A$ into a direct sum of subspaces such that $\mathcal A_i\cdot \mathcal A_j\subseteq \mathcal A_{i+j}$ for all $i,j\in\Z$.
In this case, we refer to the sequence $\{\mathcal A_i\}_{i\in\Z}$ as a $\Z$-{\it grading} of $\mathcal A$.  
For $i\in\Z$, we refer to $\mathcal A_i$ as the  {\it $i$-homogeneous component} of $\mathcal A$ and refer to $i$ as the  {\it degree} of $\mathcal A_i$.  An element of $\mathcal A$ is said to be  {\it homogeneous with degree} $i$ whenever it is contained in $\mathcal A_i$.

We now turn our attention to a basis for $\U$.  
\begin{lem}{\rm\label{lemma:UPBW}\cite[Theorem 2.13]{mazorchuk}}
The elements
\begin{gather*}
e^i h^j f^k
\qquad \quad
\hbox{$i,j,k\in \N$}
\end{gather*}
form a basis of $\U$.  
\end{lem}
 
The basis in Lemma \ref{lemma:UPBW} is referred to as a  \textit{Poincar\'{e}--Birkhoff--Witt basis} or \textit{PBW-basis} for short.  We can use this PBW-basis to form a $\Z$-grading of $\U$ in the following way.
For each integer $n$, let $U_n$ denote the $\F$-subspace of $\U$ spanned by
\begin{equation*}
e^i h^j f^k
\qquad \quad
\hbox{$i,j,k\in \N$ with $k-i=n$}.
\end{equation*}
The sequence $\{ U_n\}_{n\in\Z}$ is a $\Z$-grading of $\U$. With respect to this $\Z$-grading, the elements $e,h,f$ are homogeneous with degree $-1, 0, 1$ respectively.  By (\ref{eq:Lambdadef}), $\Lambda$ is homogeneous with degree 0.

By construction, for $n\in\Z$, $U_n$ has a basis consisting of the elements $e^i h^j f^k$ such that $k-i=n$.  Shortly we will display a second basis for $U_n$ which will aid us in proving the remaining theorems.  We will need the following result.

\begin{lem}\label{lemma:ef^n}
For $i\in\N$,
\begin{equation*}
e^if^i=\prod_{j=1}^i\left(\Lambda -\frac{\left(h-2j+2\right)\left(h-2j\right)}{4}\right).
\end{equation*}
\end{lem}
\begin{proof}
Assume $i\geq 1$; otherwise the result is trivial.
It follows from (\ref{eq:Udef}) that $eh=(h-2)e$ and thus $e^{i-1}h=(h-2(i-1))e^{i-1}$.
Using these facts along with  (\ref{eq:Lambdadef}), we see that
\begin{align*}
e^if^i&=e^{i-1}eff^{i-1}\\
&=e^{i-1} \left(\Lambda-\frac{h(h-2)}{4}\right) f^{i-1}\\
&= \left(\Lambda-\frac{(h-2i+2)(h-2i)}{4}\right) e^{i-1}f^{i-1}
\end{align*}
The result follows from these comments and induction on $i$.
\end{proof}

We are now ready to give a second basis for the $\F$-vector space $U_n$.

\begin{lem}\label{lemma:Ubasis_Chev}
For each $n\in\N$, both of the following hold:
\begin{enumerate}
\item the $\F$-vector space $U_n$ has a basis
\begin{equation*}
\Lambda^i h^j f^n
\qquad \quad
i,j\in \N,
\end{equation*}
  
\item the $\F$-vector space $U_{-n}$ has a basis
\begin{equation*}
\Lambda^i h^j e^n
\qquad \quad
 i,j\in \N.
\end{equation*}
\end{enumerate}
\end{lem}
\begin{proof}
We first show (i). Let $H$ denote the subalgebra of $\U$ generated by $h$.  Observe that the elements $\{h^j\}_{j\in\N}$ are linearly independent by 
Lemma \ref{lemma:UPBW} and so it follows that for each $i\in\N$, $\{(h+2i)^j\}_{j\in\N}$ is a basis for $H$.  
By Lemma \ref{lemma:UPBW}, the sum $U_n=\sum_{k =0}^\infty e^k  H f^{n+k }$ is direct.  
We have $e H=He$ since $eh=(h-2)e$.  Similarly, we have that $fH=Hf$.  
Pick an integer $i\in\N$. By Lemma \ref{lemma:ef^n} and induction on $i$, we find that $\Lambda^{i}\in\sum_{k  =0}^{i} e^k   H f^k  $ and $\Lambda^{i}-e^{i}f^{i}\in \sum_{k  =0}^{i-1} e^k   H f^k  $. 
 It follows from these comments that  
  for the above $i$ and all $j\in \N$,  
 both
  $\Lambda^{i}h^jf^n\in\sum_{k  =0}^{i} e^k   H f^{n+k  }$ and
\begin{equation*}
\Lambda^{i}h^jf^n-   e^{i} (h+2i)^jf^{n+i}        \in\sum_{k  =0}^{i-1} e^k   H f^{n+k  }.
\end{equation*}
Observe that since $\{e^{i} h^j f^{n+i} |i,j\in\N\}$ is a basis for $U_n$, $\{e^{i} (h+2i)^j f^{n+i} |i,j\in\N\}$ is also a basis for $U_n$.
The result follows from these comments.

Part (ii) can be similarly shown.
\end{proof}

We now consider the $\Z$-grading $\{U_n\}_{n\in\Z}$ from an alternative point of view.    
For notational convenience, we let
\begin{gather}
\nu_x=
\frac{f}{2},
\qquad
\quad
\nu_z=
\frac{e}{2}.\label{eq:nu}
\end{gather}
Observe that it follows from \eqref{eq:isowhat} that 
\begin{gather}
\nu_x=-\frac{1}{2}(x+y),
\qquad
\quad
\nu_z=\frac{1}{2}(y+z).\label{eq:nu2}
\end{gather} 
 Recall from (\ref{eq:Uequit}) that $y=h/2$.
Note that $\nu_x,y,\nu_z$ are homogeneous with degree $1, 0,-1$ respectively. 
Further observe that $\nu_x,y,\nu_z$ form a generating set for $U(\mathfrak{sl}_2)$ and that 
\begin{equation}
\qquad\qquad
[\nu_x,y]=\nu_x,\qquad\qquad
[y,\nu_z]=\nu_z,\qquad\qquad
[\nu_z,\nu_x]=\frac{y}{2}.\label{eq:ynu}
\end{equation}
We remark that the Casimir element $\Lambda$ can be expressed in the following ways:
\begin{equation}
\Lambda=4\nu_x\nu_z+y(y+1),\qquad\qquad\Lambda=4\nu_z\nu_x+y(y-1).\label{eq:Lambdanu}
\end{equation}
From this we see that both $\nu_x\nu_z$, $\nu_z\nu_x$ are homogeneous with degree 0.\\

We now give a reformulation of Lemma \ref{lemma:Ubasis_Chev} in terms of $\Lambda, y, \nu_x,\nu_z$.

\begin{lem}\label{lemma:Ubasis}
For each $n\in\N$, both of the following hold:
\begin{enumerate}
\item the $\F$-vector space $U_n$ has a basis
\begin{equation*}
\Lambda^i y^j \nu_x^n
\qquad \quad
i,j\in \N,
\end{equation*}

\item the $\F$-vector space $U_{-n}$ has a basis
\begin{equation*}
\Lambda^i y^j \nu_z^n
\qquad \quad
i,j\in \N.
\end{equation*}
\end{enumerate}
\end{lem}
\begin{proof}
This is a reformulation of Lemma \ref{lemma:Ubasis_Chev} using  (\ref{eq:Uequit}) and (\ref{eq:nu}).
\end{proof}

In the coming sections, it will be useful to consider the homogeneous components of various elements of $\U$.  With this in mind, we define the following family of maps for future use.

\begin{defn}\label{def:proj}
For $n\in\Z$, let $\pi_n:\U\to\U$ denote the $\F$-linear map such that $\left(\pi_n-I\right)U_n=0$ and $\pi_n U_m=0$ for $m\neq n$.  Thus, $\pi_n$ is the projection map from $\U$ onto $U_n$.  For $u\in\U$, we call $\pi_n(u)$ the homogenous component of $u$ of degree $n$.  
\end{defn}

\section{A $\Z$-grading of $\F[a,b,c]\otimes\U$}\label{section:FUgrad}

Recall from the Introduction that $a,b,c$ are three mutually commuting indeterminates and that $\F[a,b,c]$ denotes the $\F$-algebra of polynomials in $a,b,c$ that have all coefficients in $\F$. 
In Section \ref{section:Ugrad} we recalled a PBW-basis of $\U$ and the associated $\Z$-grading of $\U$.  
In this section, we extend the $\Z$-grading of $\U$ to a $\Z$-grading of $\FU$ and discuss the homogeneous components of certain elements of $\FU$.

It follows
from Lemma \ref{lemma:Ubasis} 
 that $\Lambda, y, \nu_x,\nu_z$ form a generating set for the $\F$-algebra $\U$.  Therefore the following elements form a generating set for the $\F$-algebra $\FU$:
\begin{equation}
1\otimes\Lambda,\qquad 1\otimes y, \qquad 1\otimes \nu_x, \qquad 1\otimes \nu_z, \qquad a\otimes 1,\qquad b\otimes 1,\qquad c \otimes 1.\label{eq:FUgens}
\end{equation}
The $\Z$-grading of $\U$ given in Section \ref{section:Ugrad} can be extended to a $\Z$-grading of $\FU$ whose homogeneous components are described as follows.  For $n\in\Z$, the $n$-homogeneous component of $\FU$ is $\F[a,b,c]\otimes U_n$.  With respect to this $\Z$-grading, the generators of $\FU$ given in \eqref{eq:FUgens} have the following degrees: 
\begin{table}[H]
\small
\centering  
\extrarowheight=3pt
\begin{tabular}{c|ccccccc}
$v$ &$1\otimes\Lambda $&$ 1\otimes y $&$  1\otimes \nu_x $&$  1\otimes \nu_z $&$  a \otimes 1 $&$  b \otimes 1 $&$  c \otimes 1$\\
\hline\hline
degree & 0&0&1&-1&0&0&0
\end{tabular}
\end{table}

\begin{lem}\label{lemma:FUbasis}
For each $n\in\N$, both of the following hold:
\begin{enumerate}
\item the $\F$-vector space $\F[a,b,c]\otimes U_n$ has a basis
\begin{equation*}
a^r b^s c^t\otimes \Lambda^i y^j \nu_x^n
\qquad \quad
i,j,r,s,t\in \N,
\end{equation*}

\item the $\F$-vector space $\F[a,b,c]\otimes U_{-n}$ has a basis
\begin{equation*}
a^r b^s c^t\otimes \Lambda^i y^j \nu_z^n
\qquad \quad
i,j,r,s,t\in \N.
\end{equation*}
\end{enumerate}
\end{lem}
\begin{proof}
The result follows from Lemma \ref{lemma:Ubasis} and the fact that $\{a^r b^s c^t\ |\ r,s,t\in\N\}$ is a basis for $\F[a,b,c]$.
\end{proof}

\begin{cor}\label{cor:extbasis}
For each $n\in\N$, both of the following hold:
\begin{enumerate}
\item the $\F[a,b,c]$-module $\F[a,b,c]\otimes U_n$ has a basis
\begin{equation*}
1\otimes \Lambda^i y^j \nu_x^n
\qquad \quad
i,j\in \N,
\end{equation*}

\item the $\F[a,b,c]$-module $\F[a,b,c]\otimes U_{-n}$ has a basis
\begin{equation*}
1\otimes\Lambda^i y^j \nu_z^n
\qquad \quad
i,j\in \N.
\end{equation*}
\end{enumerate}
\end{cor}

Shortly we will describe the homogeneous components of $A^\natural$, $B^\natural$, $C^\natural$, $D^\natural$, $\alpha^\natural$, $\beta^\natural$, $\gamma^\natural$, $\delta^\natural$.  To aid us in this task, we first express the elements $A^\natural, B^\natural$  
  in terms of $\nu_x,\nu_z$ rather than $x,z$.

\begin{lem}\label{lemma:ABCnu}
	The following equations hold in $\FU$:
		\begin{eqnarray}
			A^\natural &=& (1\otimes y+a\otimes 1)(1\otimes y+(a+1)\otimes 1)+2\otimes y\nu_x+2\left(c+a-b+1\right)\otimes\nu_x,\label{eq:Anat}\\
			B^\natural &=& (1\otimes y-b\otimes 1)(1\otimes y-(b+1)\otimes 1)-2\otimes y\nu_z-2\left(a-b-c-1\right)\otimes \nu_z.\label{eq:Bnat} 
		\end{eqnarray}
\end{lem}
\begin{proof}
 Use (\ref{eq:nu2}) to rewrite $A^\natural$ and $B^\natural$ in terms of $\nu_x,\nu_z$.
\end{proof}

Motivated by the above result, we define $R,L\in\FU$ by 
\begin{eqnarray}
R 
&=& 2\otimes y\nu_x+2\left(c+a-b+1\right)\otimes\nu_x,\label{eq:RLtheta1} 
\\
L 
&=& -2\otimes y\nu_z-2\left(a-b-c-1\right)\otimes \nu_z.\label{eq:RLtheta2} 
\end{eqnarray}
We additionally define
\begin{equation}
\theta = 1\otimes y-b\otimes 1,\qquad\qquad
\vartheta = 1\otimes y+a\otimes 1.\label{eq:RLtheta4}
\end{equation}

We now mention a few quick results concerning the elements $R,L,\theta,\vartheta$.

\begin{lem}\label{lemma:RLnonzero}
Each of $R,L,\theta,\vartheta$ is nonzero.  Moreover, the elements $R,L,\theta,\vartheta$ are homogeneous with degree $1, -1,0,0$ respectively.
\end{lem}
\begin{proof}
The result follows from Lemma \ref{lemma:FUbasis}.
\end{proof}

\begin{lem}\label{lemma:Ry}
The following equations hold in $\FU$:
\begin{equation}
[R,1\otimes y]=R,\qquad \qquad [1\otimes y,L]=L.
\end{equation}
\end{lem}
\begin{proof}
	This result follows from (\ref{eq:ynu}). 
\end{proof}

\begin{cor}\label{cor:RLthetas}
The following equations hold in $\FU$:
\begin{equation*}
[R,\theta]=R,\qquad \qquad [R,\vartheta]=R,\qquad \qquad[\theta,L]=L,\qquad \qquad[\vartheta,L]=L.
\end{equation*}
Moreover, 
\begin{gather*}
\theta R=R\left(\theta-1\otimes 1\right) ,
\qquad 
\vartheta R=R\left(\vartheta-1\otimes 1\right) ,
\qquad 
\theta L= L (\theta+ 1\otimes 1),
\qquad 
\vartheta L= L (\vartheta+ 1\otimes 1).
\end{gather*}
\end{cor}

\begin{lem}\label{lemma:RLandLR}
The following equations hold in $\FU$:
\begin{align}
RL&= (1\otimes y+(c+a-b+1)\otimes 1 ) (1\otimes y+(a-b-c)\otimes 1 ) (1\otimes y(y+1)-1\otimes \Lambda ),
\\
LR&=(1\otimes y+(c+a-b)\otimes 1 ) (1\otimes y+(a-b-c-1)\otimes 1 )(1\otimes y(y-1)-1\otimes \Lambda).
\end{align}
Moreover, each of  $RL, LR, [R,L]$ is homogeneous with degree 0.
\end{lem}
\begin{proof}
This result follows from \eqref{eq:ynu}, \eqref{eq:Lambdanu}, \eqref{eq:RLtheta1}, and \eqref{eq:RLtheta2}.
\end{proof}

\begin{lem}\label{lemma:comm}
The elements $\theta,\vartheta,RL,LR,[R,L]$ mutually commute.
\end{lem}
\begin{proof}
This result follows from  (\ref{eq:RLtheta4}) and Lemma \ref{lemma:RLandLR}. 
\end{proof}

We are now ready to describe the homogeneous components of $A^\natural$, $B^\natural$, $C^\natural$, $D^\natural$, $\alpha^\natural$, $\beta^\natural$, $\gamma^\natural$, $\delta^\natural$ and some related products.   
To aid us in displaying these results, we define the following family of maps. 
\begin{defn}\label{def:tproj}
For $n\in\Z$, let $\tilde\pi_n:\FU\to\FU$ denote the projection map from $\FU$ onto the $n$-homogeneous component of $\FU$.  That is, $\tilde\pi_n$ is the $\F$-linear map that acts as the identity on the $n$-homogeneous component of $\FU$ and acts as $0$ on all other homogeneous components of $\FU$.  Observe that for  $f\in\F[a,b,c]$ and $u\in\U$, $\tilde\pi_n\left(f\otimes u\right)=f\otimes \pi_n(u)$ and so $\tilde\pi_n=1\otimes\pi_n$. 
 \end{defn}

\begin{lem}\label{lemma:imagedegree0}
	Each of $\alpha^\natural$, $\beta^\natural$, $\gamma^\natural$, $\delta^\natural$ is homogeneous with degree 0.
\end{lem}
\begin{proof}
	The result follows from Lemma \ref{lemma:FUbasis}.
\end{proof}

\begin{lem}\label{lemma:imagedegree}
For each of the elements $A^\natural,B^\natural,C^\natural,D^\natural$, we display the $n$-homogeneous component for $-1\leq n \leq 1$.  All other homogeneous components of $A^\natural,B^\natural,C^\natural,D^\natural$ are zero.
\begin{table}[H]
\small
\centering
\extrarowheight=3pt
\begin{tabular}{c|ccc}
$v$&$\tilde\pi_{-1}(v)$\quad &$\tilde\pi_{0}(v)$& \quad $\tilde\pi_{1}(v)$\\
\hline
\hline
$A^\natural$ & $0$ & $\vartheta(\vartheta+1\otimes 1)$ & $R$\\
$B^\natural$ & $L$ & $\theta(\theta-1\otimes 1)$ & $0$\\
$C^\natural$ & $-L$ & $1\otimes\Lambda+a(a+1)\otimes 1+b(b+1)\otimes 1+c(c+1)\otimes 1-\vartheta(\vartheta+1\otimes 1)-\theta(\theta-1\otimes 1)$ & $-R$ \\
$D^\natural$ & $\vartheta L$ & $\frac{1}{2}[R,L]$ & $\theta R$\\
\end{tabular}
\end{table}
\end{lem}
\begin{proof}
The first two lines of the table are readily obtained from Lemma \ref{lemma:ABCnu} using \eqref{eq:RLtheta1}--\eqref{eq:RLtheta4}.
To obtain the third line of the table, use the above results for $A^\natural$ and $B^\natural$ along with Lemma \ref{lemma:imagedegree0} and the fact that $C^\natural=\delta^\natural-A^\natural-B^\natural$. 
To obtain the fourth line of the table, use the above results for $A^\natural$ and $B^\natural$ along with the fact that $D^\natural=\frac{1}{2}[A^\natural,B^\natural]$.  Then use Corollary \ref{cor:RLthetas}, Lemma \ref{lemma:RLandLR}, and Lemma \ref{lemma:comm} to simplify the result. 
\end{proof}

\begin{lem}\label{lemma:ABCDhom}
	For $i\in\N$, each of the following hold:
	\begin{enumerate}
		\item The homogeneous component of $(A^\natural)^i$ of degree $n$ is zero unless $0\leq n\leq i$.  Moreover, the homogeneous component of degree 0 is $\vartheta^i(\vartheta+1\otimes 1)^i$ and the homogeneous component of degree $i$ is $R^i$.
		\item The homogeneous component of $(B^\natural)^i$ of degree $n$ is zero unless $-i\leq n\leq 0$.  Moreover, the homogeneous component of degree $-i$ is $L^i$ and the homogeneous component of degree 0 is $\theta^i(\theta-1\otimes 1)^i$.
		\item The homogeneous component of $(C^\natural)^i$ of degree $n$ is zero unless $-i\leq n\leq i$.  Moreover, the homogeneous component of degree $-i$ is $(-L)^i$ and the homogeneous component of degree i is $(-R)^i$.
		\item The homogeneous component of $(D^\natural)^i$ of degree $n$ is zero unless $-i\leq n\leq i$.  Moreover, the homogeneous component of degree $-i$ is $(\prod_{j=0}^{i-1}(\vartheta-j\otimes 1))L^i$ and the homogeneous component of degree i is $R^i\prod_{j=1}^i(\theta-j\otimes 1)$.
	\end{enumerate}
\end{lem}
\begin{proof}	
	The result follows from Lemma \ref{lemma:RLnonzero}, Corollary \ref{cor:RLthetas}, Lemma \ref{lemma:RLandLR}, 
	and Lemma \ref{lemma:imagedegree}. 
\end{proof}

\begin{lem}\label{lemma:ADBhom}
For $i,j,k\in \N$, the homogeneous component of
$(A^\natural)^i (D^\natural)^j (B^\natural)^k$ of degree $n$ is zero unless $-j-k\leq n\leq i+j$. Moreover, the homogeneous component of $(A^\natural)^i (D^\natural)^j (B^\natural)^k$ of degree $i+j$ is equal to
\begin{equation}
R^{i+j}\theta^k (\theta- 1\otimes 1)^k\prod_{\ell=1}^j(\theta-\ell\otimes 1)\label{eq:homcomp1}
\end{equation}
and
the homogenenous component of $(A^\natural)^i (D^\natural)^j (B^\natural)^k$ of degree $-j-k$ is equal to
\begin{equation}
\vartheta^i(\vartheta+ 1\otimes 1)^i
\left(\prod_{\ell=0}^{j-1} \left(\vartheta-\ell\otimes 1 \right)\right)
L^{j+k}.\label{eq:homcomp2}
\end{equation}
\end{lem}
\begin{proof}
The result follows from Lemma \ref{lemma:ABCDhom}. 
\end{proof}

\section{The images of $\Omega_A$, $\Omega_B$, $\Omega_C$ under $\natural$}\label{section:Casimages}
Recall the homomorphism $\natural:\Re\to\FU$ from Section \ref{section:embed}.  In Section \ref{section:FUgrad}, we described the images of the generators of the Racah algebra $\Re$ under the homomorphsim $\natural$ in terms of their homogeneous components.
In this section, we use those results to describe the images of the Casimir elements $\Omega_A,\Omega_B,\Omega_C$ under $\natural$.
 
\begin{thm}\label{thm:Casimage}
\begin{enumerate}
\item The image of $\Omega_A$ under $\natural$ is
\begin{align*}
&(1\otimes \Lambda+a(a+1)\otimes 1-b(b+1)\otimes 1-c(c+1)\otimes 1)\cdot (a(a+1)\otimes \Lambda-b(b+1)c(c+1)\otimes 1)
\\
&\qquad-(1\otimes \Lambda+a(a+1)\otimes 1)\cdot(b(b+1)\otimes 1+c(c+1)\otimes 1).
\end{align*}

\item The image of $\Omega_B$ under $\natural$ is
\begin{align*}
&(1\otimes \Lambda+b(b+1)\otimes 1-c(c+1)\otimes 1-a(a+1)\otimes 1)\cdot (b(b+1)\otimes \Lambda-c(c+1)a(a+1)\otimes 1)
\\
&\qquad-(1\otimes \Lambda+b(b+1)\otimes 1)\cdot(c(c+1)\otimes 1+a(a+1)\otimes 1).
\end{align*}

\item The image of $\Omega_C$ under $\natural$ is
\begin{align*}
&(1\otimes \Lambda+c(c+1)\otimes 1-a(a+1)\otimes 1-b(b+1)\otimes 1)\cdot (c(c+1)\otimes \Lambda-a(a+1)b(b+1)\otimes 1)
\\
&\qquad-(1\otimes \Lambda+c(c+1)\otimes 1)\cdot(a(a+1)\otimes 1+b(b+1)\otimes 1).
\end{align*}
\end{enumerate}
\end{thm}
\begin{proof}
We first show (i).  By (\ref{eq:CasA}) and Theorem \ref{thm:natural}, the image of  $\Omega_A$ under $\natural$ is given by 
\begin{equation*}
\left(D^\natural\right)^2
+
\frac{B^\natural A^\natural C^\natural
+C^\natural A^\natural B^\natural}{2}
+ \left(A^\natural\right)^2
+B^\natural \gamma^\natural
-C^\natural \beta^\natural
-A^\natural \delta^\natural.
\label{eq:CasA2}
\end{equation*}
We begin by considering each of the products $\left(D^\natural\right)^2$, $B^\natural A^\natural C^\natural$, $C^\natural A^\natural B^\natural$, $\left(A^\natural\right)^2$, $B^\natural \gamma^\natural$, $C^\natural \beta^\natural$, $A^\natural \delta^\natural$ separately.   Due to the large number of terms in each of these products, we decompose each product into its homogeneous components.
All homogeneous components not displayed in the tables below are zero. 
For the sake of brevity, we write $\delta^\natural$ rather than the expression $\comm$ whenever possible. 

We begin with $(D^\natural)^2$.  Recall from Lemma \ref{lemma:imagedegree} that $D^\natural=\theta R+\frac{1}{2}[R,L]+\vartheta L$.  It follows from this fact along with  
Corollary \ref{cor:RLthetas} and Lemma \ref{lemma:comm}  that the nonzero homogeneous components of $(D^\natural)^2$ 
 are as follows: 
\begin{table}[H]
\small
\centering
\extrarowheight=3pt
\begin{tabular}{c|c}
$n$
&$\tilde{\pi}_{n}\left((D^\natural)^2\right)$ \\
\hline
\hline

$2$ &$\theta\left(\theta+1\otimes 1\right)R^2$
\\
$1$ & $\frac{1}{2}\theta R[R,L]+\frac{1}{2}\theta [R,L]R$\\
0 &$\frac{1}{4}[R,L]^2+\theta\left(\vartheta+1\otimes 1\right)RL+\vartheta\left(\theta-1\otimes 1\right)LR$
\\
$-1$ & 
$\frac{1}{2}\vartheta L[R,L]+\frac{1}{2}\vartheta [R,L]L$\\
$-2$ &$\vartheta\left(\vartheta-1\otimes 1\right)L^2$
\end{tabular}
\end{table}

We now consider the products $B^\natural A^\natural C^\natural$ and $C^\natural A^\natural B^\natural$. 
Using Corollary \ref{cor:RLthetas}, Lemma \ref{lemma:comm}, and Lemma \ref{lemma:imagedegree}, we find that  that the nonzero homogeneous components of $B^\natural A^\natural C^\natural$  
are as follows:
\begin{table}[H]
\small
\centering
\extrarowheight=3pt
\begin{tabular}{c|c}
$n$
& $\tilde{\pi}_{n}\left(B^\natural A^\natural C^\natural\right)$ \\
\hline
\hline
$2$ &$-\theta\left(\theta-1\otimes 1 \right)R^2$
\\
$1$ & $\theta(\theta-1\otimes 1 )\left(\delta^\natural-2(\vartheta+1\otimes 1 )^2-\theta(\theta+1\otimes 1 )\right)R-LR^2$
\\
0 &$\theta\left(\theta-1\otimes 1 \right)\vartheta\left(\vartheta+1\otimes 1 \right)\left(\delta^\natural-\vartheta(\vartheta+1\otimes 1 )-\theta(\theta-1\otimes 1 )\right)$ 
\\
& 
$-\theta\left(\theta-1\otimes 1 \right)RL+\left(\delta^\natural-2\vartheta^2-\theta\left(\theta-1\otimes 1 \right)\right)LR$
\\
$-1$ &$\big(\vartheta(\vartheta-1\otimes 1 )\left(\delta^\natural-\vartheta(\vartheta-1\otimes 1 )-(\theta-1\otimes 1 )(\theta-2\otimes 1 )\right)-\theta(\theta-1\otimes 1 )\vartheta(\vartheta+1\otimes 1 )\big)L-LRL$
\\
$-2$ &$-\vartheta\left(\vartheta-1\right)L^2$
\end{tabular}
\end{table}

It can similarly be shown that the nonzero homogeneous components of $C^\natural A^\natural B^\natural$  
 are as follows:
 \begin{table}[H]
\small
\centering
\extrarowheight=3pt
\begin{tabular}{c|c}
$n$
&$\tilde{\pi}_{n}\left(C^\natural A^\natural B^\natural\right)$ \\
\hline
\hline
$2$ &$-(\theta+1\otimes 1)\left(\theta+2\otimes 1\right)R^2$
\\
$1$ & $\theta(\theta+1\otimes 1)\left(\delta^\natural -2(\vartheta+1\otimes 1)^2-\theta(\theta-1\otimes 1)\right)R-R^2L$
\\
0 &$\theta\left(\theta-1\otimes 1\right)\vartheta\left(\vartheta+1\otimes 1\right)\left(\delta^\natural -\vartheta(\vartheta+1\otimes 1)-\theta(\theta-1\otimes 1)\right)$\qquad\qquad\qquad
\\
 &$+\left(\delta^\natural -2(\vartheta+1\otimes 1)^2-\theta(\theta-1\otimes 1)\right)RL-\theta(\theta-1\otimes 1)LR$
\\
$-1$ &$\big(\vartheta(\vartheta+1\otimes 1)\left(\delta^\natural-\vartheta(\vartheta+1\otimes 1)-\theta(\theta-1\otimes 1)\right)-\vartheta(\vartheta-1\otimes 1)(\theta-1\otimes 1)(\theta-2\otimes 1)\big)L-LRL$
\\
$-2$ &$-\vartheta\left(\vartheta-1\otimes 1\right)L^2$
\end{tabular}
\end{table}

We now consider the product $(A^\natural)^2$.  Recall from Lemma \ref{lemma:imagedegree} that $A^\natural=R+\vartheta(\vartheta+1\otimes 1)$.  It follows from this along with Corollary \ref{cor:RLthetas} that the nonzero homogeneous components of $(A^\natural)^2$  are as follows:
\begin{table}[H]
\small
\centering
\extrarowheight=3pt
\begin{tabular}{c|c}
$n$
&$\tilde{\pi}_{n}\left((A^\natural)^2\right)$ \\
\hline
\hline
$2$ &$R^2$
\\
$1$ & $2(\vartheta+1\otimes 1)^2R$
\\
0 &$\vartheta^2(\vartheta+1\otimes 1)^2$
\end{tabular}
\end{table}

We now consider the product $B^\natural \gamma^\natural$.  
Recall from Lemma \ref{lemma:imagedegree} that $B^\natural=L+\theta(\theta-1\otimes 1)$.
It follows from this fact along with Lemma \ref{lemma:imagedegree0}  
 that the nonzero homogeneous components of $B^\natural \gamma^\natural$  are as follows: 
 \begin{table}[H]
\small
\centering
\extrarowheight=3pt
\begin{tabular}{c|c}
$n$
&$\tilde{\pi}_{n}\left(B^\natural \gamma^\natural\right)$ \\
\hline
\hline
0 &$\theta(\theta-1\otimes 1)(1\otimes \Lambda- c(c+1)\otimes 1) (a(a+1)\otimes 1-b(b+1)\otimes 1)$
\\
$-1$ & $(1\otimes \Lambda- c(c+1)\otimes 1) (a(a+1)\otimes 1-b(b+1)\otimes 1)L$
\end{tabular}
\end{table}

We now consider the product $C^\natural \beta^\natural$.  It follows from Lemma \ref{lemma:imagedegree0}  
 and Lemma \ref{lemma:imagedegree}  
that the nonzero homogeneous components of $C^\natural \beta^\natural$ 
are as follows:
 \begin{table}[H]
\small
\centering
\extrarowheight=3pt
\begin{tabular}{c|c}
$n$
&$\tilde{\pi}_{n}\left(C^\natural \beta^\natural\right)$ \\
\hline
\hline
$1$ & $-(1\otimes \Lambda- b(b+1)\otimes 1) (c(c+1)\otimes 1-a(a+1)\otimes 1)R$
\\
0 &$(1\otimes \Lambda- b(b+1)\otimes 1) (c(c+1)\otimes 1-a(a+1)\otimes 1)\left(\delta^\natural-\vartheta(\vartheta+1\otimes 1)-\theta(\theta-1\otimes 1)\right)$
\\
$-1$ & $-(1\otimes \Lambda- b(b+1)\otimes 1)(c(c+1)\otimes 1-a(a+1)\otimes 1)L$
\end{tabular}
\end{table}

We now consider the product $A^\natural \delta^\natural$.  Recall from Lemma \ref{lemma:imagedegree} that 
$A^\natural=R+\vartheta(\vartheta+1\otimes 1)$.  It follows from this fact 
along with Lemma \ref{lemma:imagedegree0}   
that the nonzero homogeneous components of $A^\natural \delta^\natural$  are as follows:
\begin{table}[H]
\small
\centering
\extrarowheight=3pt
\begin{tabular}{c|c}
$n$
&$\tilde{\pi}_{n}\left(A^\natural \delta^\natural\right)$ \\
\hline
\hline
$1$ & $\delta^\natural R$
\\
0 &$\vartheta(\vartheta+1\otimes 1)\delta^\natural$
\end{tabular}
\end{table}

If we now use (\ref{eq:CasA}) and Lemma \ref{lemma:RLandLR} along with the information contained in the above tables, we see that the image of $\Omega_A$ under $\natural$ is as given in (i).

The proofs of (ii), (iii) are similar. 
\end{proof}

\section{The algebraic independence of $\theta$, $\alpha^\natural$, $\beta^\natural$, $\delta^\natural$, and a Casimir element}\label{section:algind}

Recall the homomorphism $\natural:\Re\to\FU$ from Section \ref{section:embed}.  In this section, we show that for any Casimir element $\Omega$ of $\Re$, the elements $\theta$ (or $\vartheta$), $\alpha^\natural$, $\beta^\natural$, $\delta^\natural$, $\Omega^\natural$ are algebraically independent over $\F$.  We use an approach similar to the one used in \cite[Section 8]{terwilligerUAW}.  

Let 
 $\{x_i\}_{i=1}^4$ 
denote mutually commuting indeterminates and let $\F[x_1,x_2,x_3,x_4]$ denote the $\F$-algebra consisting of the polynomials in $\{x_i\}_{i=1}^4$ that have all coefficients in $\F$.  
We define the following elements in $\F[x_1,x_2,x_3,x_4]$:
\begin{align}
y_1&=x_1^2x_2+x_1x_2^2-x_1x_2x_3-x_1x_2x_4-x_1x_3x_4-x_2x_3x_4+x_3^2x_4\label{eq:polys1}
\\ &\qquad +x_3x_4^2-x_1x_3-x_1x_4-x_2x_3-x_2x_4,\nonumber\\
y_2&=x_1x_3-x_1x_4-x_2x_3+x_2x_4,\\
y_3&=x_1x_4-x_1x_2-x_3x_4+x_2x_3,\\
y_4&=x_1+x_2+x_3+x_4.\label{eq:polys4}
\end{align}

\begin{lem}\label{lemma:algind1}
The elements $\{y_i\}_{i=1}^4$ in {\rm (\ref{eq:polys1})--(\ref{eq:polys4})} are algebraically independent over $\F$.
\end{lem}
\begin{proof}
The following is a basis for the $\F$-vector space $\F[x_1,x_2,x_3,x_4]$:
\begin{equation}
x_1^hx_2^ix_3^jx_4^k\qquad\qquad h,i,j,k\in\N.\label{eq:polybasis}
\end{equation}
We call an element $x_1^hx_2^ix_3^jx_4^k$ in the basis (\ref{eq:polybasis}) a {\it monomial}.  We define the {\it rank} of this monomial to be $5h+i+3j+2k$.  For example, consider the monomials that comprise $y_2$.  The monomials $x_1x_3$, $x_1x_4$, $x_2x_3$, $x_2x_4$ have rank 8, 7, 4, 3 respectively.  

To prove the lemma, it suffices to show that the following elements are linearly independent over $\F$:
\begin{equation}
y_1^ry_2^sy_3^ty_4^u\qquad\qquad r,s,t,u\in\N.\label{eq:polybasis2}
\end{equation}
Given integers $r,s,t,u\in\N$, we write $y_1^ry_2^sy_3^ty_4^u$ as a linear combination of monomials:
\begin{align*}
y_1^ry_2^sy_3^ty_4^u&=\left(x_1^2x_2+x_1x_2^2-x_1x_2x_3-x_1x_2x_4-x_1x_3x_4-x_2x_3x_4+x_3^2x_4+x_3x_4^2\right. \\ &\qquad\left. -x_1x_3-x_1x_4-x_2x_3-x_2x_4\right)^r
 \left(x_1x_3-x_1x_4-x_2x_3+x_2x_4\right)^s \\
&\qquad \times \left(x_1x_4-x_1x_2-x_3x_4+x_2x_3\right)^t
\left(x_1+x_2+x_3+x_4\right)^u\\
&=\left(x_1^2x_2\right)^r\left(x_1x_3\right)^s\left(x_1x_4\right)^tx_1^u+\text{ sum of monomials that have lower rank}\\
&=x_1^{2r+s+t+u}x_2^rx_3^sx_4^t+\text{ sum of monomials that have lower rank.}
\end{align*}
We call the monomial $x_1^{2r+s+t+u}x_2^rx_3^sx_4^t$ the \textit{leading monomial} of $y_1^ry_2^sy_3^ty_4^u$ since it is the highest rank monomial with nonzero coefficient.  Given a monomial $x_1^hx_2^ix_3^jx_4^k$ in the basis (\ref{eq:polybasis}), consider the following system of linear equations in the unknowns $r,s,t,u$:
\begin{equation*}
2r+s+t+u=h,\qquad r=i,\qquad s=j,\qquad t=k.
\end{equation*}
This system has a unique solution:
\begin{equation*}
r=i,\qquad s=j,\qquad t=k,\qquad u=h-2i-j-k.
\end{equation*}
Therefore $x_1^hx_2^ix_3^jx_4^k$ is the leading monomial for exactly one element of (\ref{eq:polybasis2}).  By these comments, the elements (\ref{eq:polybasis2}) are linearly independent over $\F$.  The result follows.
\end{proof}

\begin{lem}\label{lemma:algind2}
For any Casimir element $\Omega$ of $\Re$, the following hold:
\begin{enumerate}
\item
 The elements $\theta$, $\Omega^\natural$, $\alpha^\natural$, $\beta^\natural$, $\delta^\natural$ 
 are algebraically independent over $\F$,
\item
 The elements $\vartheta$, $\Omega^\natural$, $\alpha^\natural$, $\beta^\natural$, $\delta^\natural$ 
 are algebraically independent over $\F$.
  \end{enumerate}
\end{lem}
\begin{proof}
 We first show (i).  Recall from Section \ref{section:racah} that $\Omega-\Omega_A$ can be expressed as a polynomial in $\alpha,\beta,\delta$.  Thus, it suffices to show that $\theta$, $\Omega_A^\natural$, $\alpha^\natural$, $\beta^\natural$, $\delta^\natural$ are algebraically independent over $\F$.

By Lemma \ref{lemma:FUbasis}, the following are algebraically independent over $\F$:
\begin{equation*}
1\otimes y,\qquad 1\otimes\Lambda, \qquad a\otimes 1, \qquad b\otimes 1,\qquad c\otimes 1.
\end{equation*}
Therefore the following are algebraically independent over $\F$:
\begin{equation*}
1\otimes y- b\otimes 1,\qquad 1\otimes\Lambda, \qquad a\otimes 1, \qquad b\otimes 1,\qquad c\otimes 1.
\end{equation*}
Consequently the following are algebraically independent over $\F$:
\begin{equation}
1\otimes y- b\otimes 1,\qquad 1\otimes\Lambda, \qquad a(a+1)\otimes 1, \qquad b(b+1)\otimes 1,\qquad c(c+1)\otimes 1.\label{eq:ind}
\end{equation}
Abbreviate the sequence (\ref{eq:ind}) by $X_0,X_1,X_2,X_3,X_4$.  Let 
\begin{align*}
Y_1&=X_1^2X_2+X_1X_2^2-X_1X_2X_3-X_1X_2X_4-X_1X_3X_4-X_2X_3X_4+X_3^2X_4\\ &\qquad +X_3X_4^2-X_1X_3-X_1X_4-X_2X_3-X_2X_4.
\end{align*}
By Lemma \ref{lemma:algind1}, the following are algebraically independent over $\F$:
\begin{gather*}
X_0,\qquad Y_1, \qquad
X_1X_3-X_1X_4-X_2X_3+X_2X_4,\\
X_1X_4-X_1X_2-X_3X_4+X_2X_3,\qquad
X_1+X_2+X_3+X_4.
\end{gather*}
The above five elements are
\begin{equation}
\theta,\qquad \Omega_A^\natural,\qquad \alpha^\natural,\qquad\beta^\natural,\qquad\delta^\natural,\label{eq:algind2-A}
\end{equation}
respectively.  Thus, the elements $\theta$,  $\Omega_A^\natural$,  $\alpha^\natural$, $\beta^\natural$, $\delta^\natural$ are algebraically independent over $\F$. The result follows.

Part (ii) can also be similarly shown.
\end{proof}

\section{The injectivity of $\natural$}\label{section:injectivity}
Recall the homomorphism $\natural:\Re\to\FU$ from Section \ref{section:embed}.  In this section, we state and prove the third main result of the paper, which states that $\natural$ is injective and thus provides an embedding of $\Re$ into $\FU$.

Before doing so, we mention that the $\F$-algebra $\FU$ contains no zero divisors.  This follows from the fact that neither $\F[a,b,c]$ nor $\U$ contains any zero divisors \cite[Corollary 9.8]{carter}.  We will use this fact in the proof of Theorem \ref{thm:inj}.

\begin{thm}\label{thm:inj}
The homomorphism $\natural$ given in Theorem \ref{thm:natural} is injective.
\end{thm}

\begin{proof}
Let $X$ denote an element in the kernel of $\natural$.  We show that $X=0$.  Let $\Omega\in\{\Omega_A,\Omega_B,\Omega_C\}$ and recall the basis of $\Re$ given in \eqref{eq:basisURA2}. Let $S=S(X)$ denote the 7-tuples $(i,j,k,\ell,r,s,t)$ of nonnegative integers such that $j\in\{0,1\}$ and
\begin{equation*}
\langle X, A^i D^j B^k \Omega^\ell \alpha^r \beta^s \gamma^t\rangle\neq 0.
\end{equation*}
By construction,
\begin{equation}
X=\sum_{(i,j,k,\ell,r,s,t)\in S} \langle X, A^i D^j B^k \Omega^\ell \alpha^r \beta^s \gamma^t\rangle A^i D^j B^k \Omega^\ell \alpha^r \beta^s \gamma^t.\label{eq:X}
\end{equation}
Applying the homomorphism $\natural$ to each side of (\ref{eq:X}) we see that
\begin{equation}
0=\sum_{(i,j,k,\ell,r,s,t)\in S} \langle X, A^i D^j B^k \Omega^\ell \alpha^r \beta^s \gamma^t\rangle \left(A^\natural\right)^i \left(D^\natural\right)^j \left(B^\natural\right)^k \left(\Omega^\natural\right)^\ell \left(\alpha^\natural\right)^r \left(\beta^\natural\right)^s \left(\delta^\natural\right)^t. \label{eq:Xnat}
\end{equation}
For a 7-tuple $(i,j,k,\ell,r,s,t)\in S$, we define its {\it height} to be $i+j$ and its {\it depth} to be $j+k$.  We let $S_n^+$ denote the set of elements of $S$ which have height $n$ and let $S_n^-$ denote the set of elements of $S$ which have depth $n$. 

Assume that $X\neq 0$ so that $S$ is nonempty.  Since $\{S_n^+\}_{n=0}^\infty$ (resp. $\{S_n^-\}_{n=0}^\infty$) is a partition of $S$,
$\{S_n^+\}_{n=0}^\infty$ (resp. $\{S_n^-\}_{n=0}^\infty$) are not all empty.
By construction, $S$ is finite, so only finitely many of $\{S_n^+\}_{n=0}^\infty$ (resp. $\{S_n^-\}_{n=0}^\infty$) are nonempty.
Let $N=\max\{n|S_n^+\neq 0\}$ and $M=\max\{n|S_n^-\neq 0\}$.  By construction both $S_N^+$ and $S_M^-$ are nonempty.
  We now split the argument into two cases.

Case $N\leq M$:
Recall the projection map $\tilde{\pi}_{-M}$ from Definition \ref{def:tproj}.  Apply $\tilde{\pi}_{-M}$ to each side of (\ref{eq:Xnat}).  Let $(i,j,k,\ell,r,s,t)\in S$ and consider the corresponding summand in (\ref{eq:Xnat}).  The image of this summand under $\tilde{\pi}_{-M}$ can be evaluated using  Lemma \ref{lemma:ADBhom}.  Observe that the image is nonzero only when $(i,j,k,\ell,r,s,t)\in S_M^-$.  Using \eqref{eq:Xnat} and Lemma \ref{lemma:ADBhom}, we find that
\begin{equation*}
0=\sum_{(i,j,k,\ell,r,s,t)\in S_M^-} \langle X, A^i D^j B^k \Omega^\ell \alpha^r \beta^s \gamma^t\rangle
\vartheta^i(\vartheta+ 1\otimes 1)^i
\left(\prod_{m=0}^{j-1} (\vartheta-m\otimes 1)\right)
\left(\Omega^\natural\right)^\ell \left(\alpha^\natural\right)^r \left(\beta^\natural\right)^s \left(\delta^\natural\right)^t L^M.
\end{equation*}
By Lemma \ref{lemma:RLnonzero}, $L\neq 0$.  Recall that $\F[a,b,c]\otimes \U$ contains no zero divisors.  Therefore
\begin{equation}
0=\sum_{(i,j,k,\ell,r,s,t)\in S_M^-} \langle X, A^i D^j B^k \Omega^\ell \alpha^r \beta^s \gamma^t\rangle
\vartheta^i(\vartheta+ 1\otimes 1)^i
\left(\prod_{m=0}^{j-1} (\vartheta-m\otimes 1)\right)
\left(\Omega^\natural\right)^\ell \left(\alpha^\natural\right)^r \left(\beta^\natural\right)^s \left(\delta^\natural\right)^t.\label{eq:image}
\end{equation}
Consider the equation \eqref{eq:image} above.  Recall that $S_M^-$ is nonempty and that by construction,\\ $\langle X, A^i D^j B^k \Omega^\ell \alpha^r \beta^s \gamma^t\rangle \neq 0$ for all $(i,j,k,\ell,r,s,t)\in S_M^-$.
Consider $(i,j,k,\ell,r,s,t)\in S_M^-$.  Recall that $j\in\{0,1\}$, $i+j\leq N$, and $j+k=M$.  By assumption, $N\leq M$.  Therefore $i+j\leq M$.  For these constraints on $i,j,k$, the possible solutions $(i,j,k)$ are
\begin{equation*}
(0,0,M),(1,0,M),\mathellipsis,(M,0,M),(0,1,M-1),(1,1,M-1),\mathellipsis,(M-1,1,M-1).
\end{equation*}
For the above values of $(i,j,k)$, the corresponding values of $\vartheta^i(\vartheta+ 1\otimes 1)^i
\prod_{m=0}^{j-1} \left(\vartheta-m\otimes 1\right)$ are
\begin{gather*}
 1\otimes 1,\quad \vartheta(\vartheta+ 1\otimes 1),\quad \vartheta^2(\vartheta+ 1\otimes 1)^2,\mathellipsis,\  \vartheta^M(\vartheta+ 1\otimes 1)^M,\\
\vartheta,\quad \vartheta^2(\vartheta+ 1\otimes 1),\quad \vartheta^3(\vartheta+ 1\otimes 1)^2,\mathellipsis, \ \vartheta^M(\vartheta+ 1\otimes 1)^{M-1}.
\end{gather*}
The above list and the sequence $\{\vartheta^n\}_{n=0}^{2M}$ are both bases for the $\F$-vector space consisting of polynomials in $\vartheta$ of degree at most $2M$.
With these comments in mind, (\ref{eq:image}) gives a nontrivial $\F$-linear dependency among
\begin{equation*}
\vartheta^n \left(\Omega^\natural\right)^\ell \left(\alpha^\natural\right)^r \left(\beta^\natural\right)^s \left(\delta^\natural\right)^t \qquad\qquad n,\ell,r,s,t\in\N, \qquad n\leq 2M.
\end{equation*}
The above linear dependency contradicts Lemma \ref{lemma:algind2}.

Case $M\leq N$: 
 Apply $\tilde{\pi}_{N}$ to each side of (\ref{eq:Xnat}).  Let $(i,j,k,\ell,r,s,t)\in S$ and consider the corresponding summand in (\ref{eq:Xnat}).  The image of this summand under $\tilde{\pi}_{N}$ can be evaluated using  Lemma \ref{lemma:ADBhom}.  Observe that the image is nonzero only when $(i,j,k,\ell,r,s,t)\in S_N^+$.  Using \eqref{eq:Xnat} and Lemma \ref{lemma:ADBhom}, we find that
\begin{equation*}
0=R^{N}\sum_{(i,j,k,\ell,r,s,t)\in S_N^+} \langle X, A^i D^j B^k \Omega^\ell \alpha^r \beta^s \gamma^t\rangle
\theta^k (\theta- 1\otimes 1)^k\left(\prod_{m=1}^j(\theta-m\otimes 1)\right)
 \left(\Omega^\natural\right)^\ell \left(\alpha^\natural\right)^r \left(\beta^\natural\right)^s \left(\delta^\natural\right)^t.
\end{equation*}
 By Lemma \ref{lemma:RLnonzero}, $R\neq 0$.  Recall that $\F[a,b,c]\otimes \U$ contains no zero divisors.    It then follows that
\begin{equation}
0=\sum_{(i,j,k,\ell,r,s,t)\in S_N^+} \langle X, A^i D^j B^k \Omega^\ell \alpha^r \beta^s \gamma^t\rangle
\theta^k (\theta- 1\otimes 1)^k\left(\prod_{m=1}^j(\theta-m\otimes 1)\right)
 \left(\Omega^\natural\right)^\ell \left(\alpha^\natural\right)^r \left(\beta^\natural\right)^s \left(\delta^\natural\right)^t.\label{eq:image2}
\end{equation}
Consider the equation \eqref{eq:image2} above.  Recall that $S_N^+$ is nonempty and that by construction,\\ $\langle X, A^i D^j B^k \Omega^\ell \alpha^r \beta^s \gamma^t\rangle \neq 0$ for all $(i,j,k,\ell,r,s,t)\in S_N^+$.
Consider $(i,j,k,\ell,r,s,t)\in S_N^+$.  Recall that $j\in\{0,1\}$, $i+j=N$, and $j+k\leq M$.  By assumption, $M\leq N$.  Therefore $j+k\leq N$.  For these constraints on $i,j,k$, the possible solutions $(i,j,k)$ are
\begin{equation*}
(N,0,0),(N,0,1),\mathellipsis,(N,0,N),(N-1,1,0),(N-1,1,1),\mathellipsis,(N-1,1,N-1).
\end{equation*}
For the above values of $(i,j,k)$ the corresponding values of $\theta^k (\theta- 1\otimes 1)^k\prod_{m=1}^j(\theta-m\otimes 1)$ are
\begin{gather*}
 1\otimes 1,\quad\theta(\theta+ 1\otimes 1),\quad\theta^2(\theta+ 1\otimes 1)^2,\mathellipsis, \ \theta^N(\theta+ 1\otimes 1)^N,\\
\theta- 1\otimes 1,\quad\theta(\theta- 1\otimes 1)^2,\quad\theta^2(\theta- 1\otimes 1)^3,\mathellipsis, \ \theta^{N-1}(\theta- 1\otimes 1)^{N}.
\end{gather*}
The above list and the sequence $\{\theta^n\}_{n=0}^{2N}$ are both bases for the $\F$-vector space consisting of polynomials in $\theta$ of degree at most $2N$.
With these comments in mind, (\ref{eq:image2}) gives a nontrivial $\F$-linear dependency among
\begin{equation*}
\theta^n \left(\Omega^\natural\right)^\ell \left(\alpha^\natural\right)^r \left(\beta^\natural\right)^s \left(\delta^\natural\right)^t \qquad\qquad n,\ell,r,s,t\in\N, \qquad n\leq 2N.
\end{equation*}
The above linear dependency contradicts Lemma \ref{lemma:algind2}.

In each case we reach a contradiction under the assumption that $X\neq 0$.  Therefore $X=0$ and hence $\natural$ is injective.
\end{proof}

We conclude the paper with a quick application of Theorem \ref{thm:inj}. 

\begin{lem}
The $\F$-algebra $\Re$ contains no zero divisors.
\end{lem}
\begin{proof}
The result follows from Theorem \ref{thm:inj} and the fact that $\FU$ contains no zero divisors. 
\end{proof}

\bibliographystyle{amsplain}
\bibliography{master}

\end{document}